\documentclass[11pt]{amsart}
\usepackage[marginparwidth=0pt,margin=20truemm]{geometry} 

\usepackage{amssymb}
\usepackage{amsmath}
\usepackage{amsthm}
\usepackage{color}
\usepackage{graphicx}

\DeclareMathOperator{\Rep}{Re}
\DeclareMathOperator{\Imp}{Im}

\DeclareMathOperator{\Den}{Den}

\newcommand{\R}{\mathbb {R}}
\newcommand{\C}{\mathbb {C}}
\newcommand{\Z}{\mathbb {Z}}
\newcommand{\Q}{\mathbb {Q}}

\newcommand{\CHF}{{}_1F_1}
\newcommand{\wC}{\widetilde{C}}

\numberwithin{equation}{section}
\allowdisplaybreaks[4]

\theoremstyle{plain}
\newtheorem{Lemma}{Lemma}[section]
\newtheorem{Theorem}[Lemma]{Theorem}
\newtheorem{Proposition}[Lemma]{Proposition}
\newtheorem{Corollary}[Lemma]{Corollary}

\theoremstyle{definition}

\newtheorem{Remark}[Lemma]{Remark}

\begin{document}

\title{An asymptotic property on a reciprocity law for the Bettin--Conrey cotangent sum}
\author{Hirotaka Akatsuka \and Yuya Murakami}
\address{%
Otaru University of Commerce,
3--5--21, Midori, Otaru, Hokkaido, 047--8501,
Japan.
}
\email{akatsuka@res.otaru-uc.ac.jp}
\address{%
	Faculty of Mathematics, Kyushu University
	744, Motooka, Nishi-ku, Fukuoka, 819--0395, Japan.}
\email{murakami.yuya.896@m.kyushu-u.ac.jp}


\begin{abstract}
In 2013 Bettin and Conrey have introduced a cotangent sum $c \colon \mathbb{Q}_{>0}\to \mathbb{R}$, which can be regarded as a variant of the Dedekind sum. They have discovered that the cotangent sum satisfies a kind of reciprocity laws. Roughly speaking, the reciprocity law for $c(x)$ means that there is a relation between $c(x)$ and $c(1/x)$ modulo holomorphic functions. Furthermore they have investigated Taylor coefficients $g_n$ of the implicit holomorphic function, which appears in the reciprocity law for $c(x)$, at $x=1$. As a result, they have obtained an asymptotic formula for $g_n$ as $n\to\infty$. In this paper we improve it to an asymptotic series expansion. This resolves a conjecture by Zagier. A new ingredient of this paper is to use the confluent hypergeometric function of the second kind.
\end{abstract}


\maketitle

\section{Introduction}\label{sec:Intro}
Let $h$, $k\in\Z_{>0}$ with $\gcd(h,k)=1$.
In \cite{BeCo1} Bettin and Conrey introduced a cotangent sum
defined by
\[
 c\left(\frac{h}{k}\right)
 =-\sum_{a=1}^{k-1}\frac{a}{k}\cot\left(\frac{\pi ha}{k}\right).
\]
For $x\in\Q_{>0}$ we set
\begin{equation}\label{eq:funcg}
 g(x)=xc(x)+c\left(\frac{1}{x}\right)-\frac{1}{\pi\Den(x)},
\end{equation}
where $\Den(x)\in\Z_{\geq 1}$ is the denominator of $x$.
Bettin and Conrey have proved that $g(x)$ is continued to
$\C\setminus\R_{\leq 0}$ as a holomorphic function.
They regarded this property as a kind of reciprocity laws for $c(x)$.
Because of this property together with the periodicity $c(x+1)=c(x)$,
the function $c(x)$ is a prototype of
the quantum modular form proposed by Zagier \cite{Zag1}.
See also \cite[Chapter 21]{BFOR} for quantum modular forms.

As has been obtained by Bettin--Conrey in \cite[\S 3]{BeCo1},
the reciprocity formula for $c(x)$ is closely related to the weight $1$
holomorphic Eisenstein series $E_1(z)$ with respect to the full modular group.
Here we explain this briefly. 
The function $E_1(z)$ is defined by
\[
E_1(z)=1-4\sum_{m=1}^{\infty}d(m)e^{2\pi imz}
\]
for $\Imp(z)>0$, where $d(m)$ is the number of positive divisors for $m$.
While $E_1(z)$ is not a modular form in the strict sense,
it has a flavor of modularity.
To see this, we set $\psi(z)=E_1(z)-(1/z)E_1(-1/z)$.
Then $\psi(z)$ is originally defined in $\Imp(z)>0$ and is
analytically continued to $z\in\C\setminus\R_{\leq 0}$.
Namely, $\psi(z)$ has a nice property compared with a single $E_1(z)$.
Moreover, $\psi(z)$ essentially coincides with (1.1), that is,
$\psi(x)=xg(x)/(2i)$ holds on $x\in\Q_{>0}$.
We also note that $\psi(z)$ satisfies a three term relation
in the theory of period functions by Lewis--Zagier \cite{LeZa}.
We can find generalizations of the above results:
see \cite{BeCo2} for general weights and \cite{Fo}
for twisted Eisenstein series and corresponding cotangent sums.

Bettin and Conrey investigated more detailed properties of $g(x)$.
We write the Taylor expansion of $g(x)$ at $x=1$ as
\begin{equation}\label{eq:gTC}
 g(x)=\frac{1}{\pi}\sum_{n=0}^{\infty}(-1)^n g_n (x-1)^n.
\end{equation}
As was shown by Bettin and Conrey, $g_0=-1$, $g_1=1/2$
and
\begin{equation}\label{eq:gnexpr}
 g_n=\frac{1}{n(n+1)}+2b_n+2\sum_{j=0}^{n-2}\binom{n-1}{j}b_{j+2}
\end{equation}
hold for any $n\in\Z_{\geq 2}$, where
\[
 b_k=\frac{B_k\zeta(k)}{k}.
\]
Here $B_k$ is the $k^{\text{th}}$ Bernoulli number and $\zeta(s)$
is the Riemann zeta-function.
In \cite[Theorem 2]{BeCo1} they also proved\footnote{%
The right-hand side of the fifth displayed formula
in \cite[page 5723]{BeCo1}
should be doubled.
Consequently, $2^{5/4}$ should be replaced by
$2^{9/4}$ in \cite[Theorem 2]{BeCo1}.
}
\begin{equation}\label{eq:BCasymptotic}
 g_n-\frac{1}{n}=2^{9/4}\pi^{3/4}n^{-3/4}e^{-2\sqrt{\pi n}}
\left(\sin\left(2\sqrt{\pi n}+\frac{3\pi}{8}\right)+o(1)\right)
\end{equation}
as $n\to\infty$.
Zagier has conjectured that (\ref{eq:BCasymptotic}) can be
improved to an asymptotic series expansion of the shape
\begin{equation}\label{eq:ZC}
 g_n-\frac{1}{n}=e^{-2\sqrt{\pi n}}
\Biggl(
\sum_{\begin{subarray}{c}
 3\leq k\leq K\\
 k\equiv 1\pmod{2}
      \end{subarray}}
C_k n^{-k/4}\sin(2\sqrt{\pi n}+D_k)+o(n^{-K/4})
\Biggr)
\end{equation}
for certain constants $C_k$, $D_k\in\R$.
See \cite[p.5724]{BeCo1} for the conjecture.

The purpose of this paper is
to resolve the conjecture (\ref{eq:ZC}) as follows:
\begin{Theorem}\label{Thm:main}
There are
$\widetilde{C}_l\in\langle\pi^{2m}:m=0,1,2,\ldots,l\rangle_{\Q}\setminus\{0\}$ 
such that
\begin{equation}\label{eq:main}
\begin{aligned}
g_n-\frac{1}{n}
=2^{9/4}\pi^{3/4}e^{-2\sqrt{\pi n}}
\sum_{l=0}^L(2\pi)^{-l/2}\wC_l n^{-\frac{l}{2}-\frac{3}{4}}
\sin\left(2\sqrt{\pi n}+\frac{\pi l}{4}+\frac{3\pi}{8}\right)
+O\left(n^{-\frac{L}{2}-\frac{5}{4}}e^{-2\sqrt{\pi n}}\right)
\end{aligned}
\end{equation}
as $n\to\infty$ for each $L\in\Z_{\geq 1}$,
where the implied constant depends only on $L$.
\end{Theorem}
\begin{Remark}\label{Rmk:main}
\begin{enumerate}
\item The asymptotic expansion (\ref{eq:ZC}) holds with
\[
 C_k=2^{9/4}\pi^{3/4}(2\pi)^{-(k-3)/4}\wC_{(k-3)/2},\phantom{MM}
D_k=\frac{\pi k}{8}.
\]
\item The numbers $\wC_l$ are given
in terms of the Hankel
notation and the Bernoulli numbers.
See (\ref{eq:deftC}) for concrete expressions.
The first six $\wC_l$ are
\begin{gather*}
 \wC_0=1,\phantom{MM}\wC_1=\frac{\pi^2}{3}+\frac{3}{16},\phantom{MM}
\wC_2=\frac{\pi^4}{18}+\frac{5\pi^2}{16}-\frac{15}{512},\\
\wC_3=\frac{\pi^6}{162}+\frac{23\pi^4}{160}+\frac{35\pi^2}{512}+
\frac{105}{8192},\\
\wC_4=\frac{\pi^8}{1944}
+\frac{137\pi^6}{4320}+\frac{973\pi^4}{5120}
-\frac{105\pi^2}{8192}-\frac{4725}{524288},\\
\wC_5=\frac{\pi^{10}}{29160}+\frac{229\pi^8}{51840}
+\frac{17365\pi^6}{193536}
+\frac{2849\pi^4}{16384}+\frac{3465\pi^2}{524288}+
\frac{72765}{8388608}.
\end{gather*}
\item The numbers $\wC_l$ are not necessarily positive for general $l$.
The smallest $l$ at which $\wC_l<0$ is $l=50$.
\item To visualize Theorem \ref{Thm:main}, we illustrate
a graph of
\begin{equation}\label{eq:mainRmk}
\begin{aligned}
&n^{5/2}\left(2^{-9/4}\pi^{-3/4}n^{3/4}e^{2\sqrt{\pi n}}\left(g_n-\frac{1}{n}\right)\right.\\
&\left.-\sum_{l=0}^4(2\pi)^{-l/2}\wC_l n^{-l/2}\sin\left(2\sqrt{\pi n}+\frac{\pi l}{4}+\frac{\pi}{8}\right)\right)
\end{aligned}
\end{equation}
in Figure \ref{Fig:main2}.
Theorem \ref{Thm:main} with $L=5$
implies that this is equal to
\begin{equation}\label{eq:sin}
 (2\pi)^{-5/2}\wC_5\sin\left(2\sqrt{\pi n}+\frac{13\pi}{8}\right)
+O(n^{-1/2})
\end{equation}
as $n\to\infty$.
\end{enumerate}
\end{Remark}
\begin{figure}\label{Fig:main2}
\centering
\includegraphics[width=17cm]{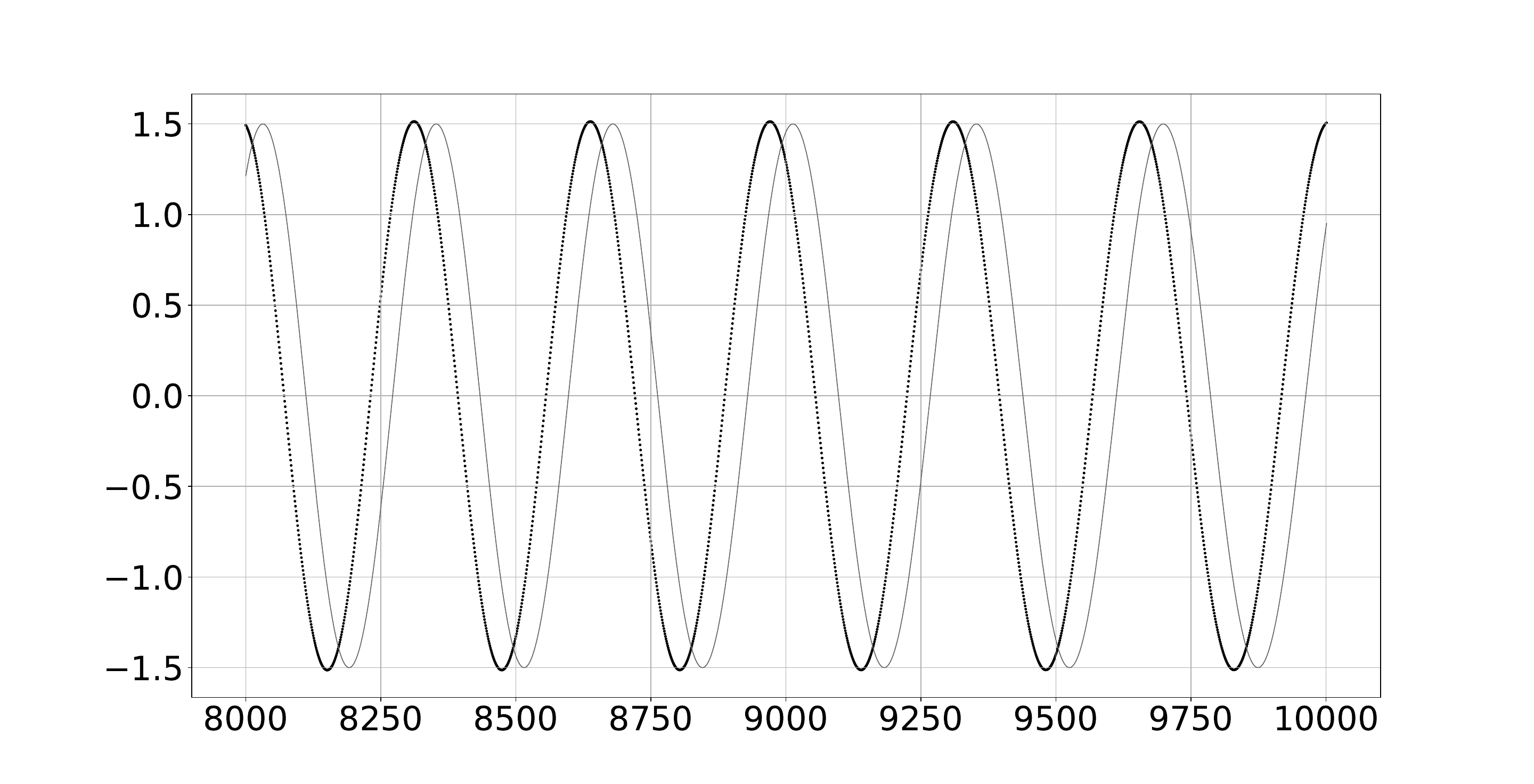}
\caption{The plot of (\ref{eq:mainRmk})
for $8000\leq n\leq 10001$.
The gray curve is the first term of (\ref{eq:sin}).
Here, $(2\pi)^{-5/2}\wC_5=1.49962\cdots$.}
\end{figure}
We explain our strategy to prove Theorem \ref{Thm:main}.
According to \cite{BeCo1}, the asymptotic analysis for $g_n-1/n$
reduces to the analysis for a certain integral.
In this paper we point out that the integral can be written
in terms of the confluent hypergeometric function of the second kind.
Applying the classical theory of large parameter problems
on the confluent hypergeometric function,
we obtain the asymptotic series expansion
(\ref{eq:main}).

This paper is organized as follows.
In \S \ref{sec:Bessel} we briefly review the definition
and properties of the $K$-Bessel function, which are frequently
used throughout this paper.
In \S \ref{sec:BCarg} we recall the discussion on $g(x)$ and $g_n$
by Bettin--Conrey \cite{BeCo1} with some additional explanations.
In \S \ref{sec:CHF} we see that an integral related to
the asymptotic behavior of $g_n-1/n$ can be expressed in terms of
the confluent hypergeometric function of the second kind.
In \S \ref{sec:PMT} we prove Theorem \ref{Thm:main} with
explicit expressions for $\wC_l$.
In \S \ref{sec:RNC} we explain numerical computations of $g_n$,
which are needed to create Figure \ref{Fig:main2}.
\subsection*{Notation and Conventions}
Throughout this paper we use the following notation and conventions:
\begin{itemize}
 \item
The $k^{\rm th}$ Bernoulli number $B_k$ is defined by
the generating function as
\[
 \frac{u}{e^u-1}=\sum_{k=0}^{\infty}\frac{B_k}{k!}u^k.
\]
 \item $\zeta(s)$ is the Riemann zeta-function.
 \item The number $\gamma$
is the Euler--Mascheroni constant, which is defined by
$\gamma=\lim_{k\to\infty}(\sum_{j=1}^k\frac{1}{j}-\log k)$.
 \item For $\alpha\in\C$ and $k\in\Z_{\geq 0}$
the Pochhammer symbol $(\alpha)_k$
is defined by
\[
 (\alpha)_k=\frac{\Gamma(\alpha+k)}{\Gamma(\alpha)}
=\begin{cases}
  1                                & \text{if $k=0$,}\\
  \alpha(\alpha+1)\cdots(\alpha+k-1) & \text{if $k\geq 1$.}
 \end{cases}
\]
\item For $\nu\in\C$ and $k\in\Z_{\geq 0}$ the Hankel symbol $(\nu,k)$
is defined by
\begin{align*}
 (\nu,k)
&=\frac{(-1)^k}{k!}\left(\tfrac{1}{2}-\nu\right)_k
\left(\tfrac{1}{2}+\nu\right)_k\\
&=\begin{cases}
   1 & \text{if $k=0$,}\\
\frac{1}{2^{2k}k!}\prod_{a=1}^k\left(4\nu^2-(2a-1)^2\right) & \text{if $k\geq 1$.}
 \end{cases}
\end{align*}
 \item For $a\in\R_{>0}$ and $w\in\C$ we determine the branch of
$a^w$ by $a^w=\exp(w\log a)$ with $\log a\in\R$
unless otherwise specified.
We also determine the branch of $(2\pi i)^w$
by $(2\pi i)^w=(2\pi)^w e^{\pi iw/2}$.
In particular, $(2\pi i)^{1/2}=\sqrt{2\pi i}=\sqrt{\pi}(1+i)$.
\item Let $S$ be a finite subset $S$ of $\C$.
Then the $\Q$-span of $S$, which is regarded as a $\Q$-subspace
of $\C$, is denoted by $\langle S\rangle_{\Q}$.
 \item Let $X$ be a set. Let $f$ be a complex-valued function
and $g$ be a positive-valued function on $X$.
We write $f(x)=O(g(x))$ or $f(x)\ll g(x)$ if there exists
$C>0$ such that $|f(x)|\leq Cg(x)$ for any $x\in X$.
When the constant $C$ may depend on a parameter $\alpha$,
we write $f(x)=O_{\alpha}(g(x))$
or $f(x)\ll_{\alpha}g(x)$.
For positive-valued functions $f$ and $g$
we write $f(x)\asymp g(x)$
if $f(x)\ll g(x)$ and $g(x)\ll f(x)$ are satisfied.
\end{itemize}
\subsection*{Acknowledgments}
The first author was supported by JSPS KAKENHI Grant Number JP1903392.
The second author was supported by JSPS KAKENHI Grant Number JP23KJ1675.
This work was also supported by
the Research Institute for Mathematical Sciences,
an International Joint Usage/Research Center located in Kyoto University.
We would like to thank Toshiki Matsusaka and Takuya Yamauchi for helpful comments
and discussions.
\section{The $K$-Bessel functions}\label{sec:Bessel}
In this section we recall the $K$-Bessel function.
Readers may proceed to the next section
and go back to this section as necessary.

For $\nu\in\C$ the $I$-Bessel function $I_{\nu}(z)$ is defined by
\begin{equation}\label{eq:defIBessel}
 I_{\nu}(z)
=\sum_{k=0}^{\infty}\frac{(z/2)^{\nu+2k}}{k!\Gamma(k+\nu+1)},
\end{equation}
where $\arg(z/2)\in(-\pi,\pi)$.
By the ratio test, the radius of convergence is $\infty$.
Therefore, $I_{\nu}(z)$ is a single-valued holomorphic function in
$z\in\C\setminus\R_{\leq 0}$.
For $\nu\in\C\setminus\Z$ the $K$-Bessel function $K_{\nu}(z)$
is defined by
\[
 K_{\nu}(z)=\frac{\pi}{2}\frac{I_{-\nu}(z)-I_{\nu}(z)}{\sin(\pi\nu)}.
\]
For $n\in\Z$ we define $K_n(z)$ by
\[
 K_n(z)=\lim_{\nu\to n}K_{\nu}(z).
\]
By definition we easily see
\begin{equation}\label{eq:KBesselef}
 K_{\nu}(z)=K_{-\nu}(z)
\end{equation}
for $\nu\in\C$ and $z\in\C\setminus\R_{\leq 0}$.
If $\nu\in\R$, then $I_{\nu}(z)$ is real on $z\in\R_{>0}$,
so that $K_{\nu}(z)$ is real on $z\in\R_{>0}$.
In other words,
\begin{equation}\label{eq:RPBessel}
 \overline{K_{\nu}(\overline{z})}=K_{\nu}(z)
\end{equation}
holds for $\nu\in\R$ and $z\in\C\setminus\R_{\leq 0}$.

We recall bounds and asymptotic formulas for $K_{\nu}(z)$.
We know
\begin{equation}\label{eq:Knuto0}
 K_{\nu}(z)\ll_{\nu,\delta}
\begin{cases}
 |z|^{-|\Rep(\nu)|} & \text{if $\nu\in\C\setminus\{0\}$,}\\
 \log(2/|z|)        & \text{if $\nu=0$}
\end{cases}
\end{equation}
on $|z|\leq 1$ with $|\arg z|\leq\pi-\delta$ for any fixed $\delta>0$.
This is immediate by definition when $\nu\in\C\setminus\Z$.
In the case $\nu\in\Z$ see \cite[eq.~(5.7.12) in p.111]{Leb}.
On the other hand, $K_{\nu}(z)$ has the following asymptotic formula
as $|z|\to\infty$ on $|\arg z|\leq\pi-\delta$ for each
$\delta>0$ and $J\in\Z_{\geq 0}$:
\begin{equation}\label{eq:Knutoinf}
 K_\nu(z)=
\left(
\frac{\pi}{2z}
\right)^{1/2}e^{-z}
\left(
\sum_{j=0}^J(\nu,j)(2z)^{-j}+O_{\nu,\delta,J}\left(|z|^{-J-1}\right)
\right).
\end{equation}
See \cite[eq.~(5.11.9) in p.123]{Leb}.

Next we recall integral expressions for $K_{\nu}(z)$.
When $\Rep(\nu)>-1/2$,
we have
\begin{equation}\label{eq:K1IE}
 K_{\nu}(z)=\frac{\sqrt{\pi}}{\Gamma\left(\nu+\frac{1}{2}\right)}
\left(\frac{z}{2}\right)^{\nu}
\int_1^{\infty}e^{-zt}(t^2-1)^{\nu-\frac{1}{2}}dt
\end{equation}
in $|\arg z|<\pi/2$:
see \cite[eq.~(5.10.24) in p.119]{Leb}.
For $\nu\in\C$, the function $K_{\nu}(z)$ has another integral expression:
\begin{equation}\label{eq:KnuIE-2}
 K_{\nu}(z)=\frac{1}{2}\left(\frac{z}{2}\right)^{\nu}
\int_{0}^{\infty}e^{-u-\frac{z^2}{4u}}u^{-\nu-1}du
\end{equation}
in $\{z\in\C\setminus\{0\}:|\arg z|<\pi/4\}$.
See \cite[eq.~(5.10.25) in p.119]{Leb}.
When $\Rep(\nu)<0$, the integral converges absolutely and uniformly
on $\{z\in\C\setminus\{0\}:|\arg z|\leq\pi/4\}$,
so that the formula (\ref{eq:KnuIE-2})
remains true on $\{z\in\C\setminus\{0\}:\arg z=\pm\pi/4\}$
by continuity.

We recall the Mellin transform for $K_{\nu}(z)$.
According to \cite[formula (26) in p.331]{ErdTIT1}
we have
\begin{equation}\label{eq:MKBF}
 \int_0^{\infty}K_{\nu}(u)u^{s}\frac{du}{u}
=2^{s-2}
\Gamma\left(\frac{s+\nu}{2}\right)\Gamma\left(\frac{s-\nu}{2}\right)
\end{equation}
in $\Rep(s)>|\Rep(\nu)|$.
We give a sketch of the proof here.
We may assume $\Rep(\nu)>-1/2$ without loss of generality
because of (\ref{eq:KBesselef}).
We insert (\ref{eq:K1IE}) into the left-hand side of (\ref{eq:MKBF})
and apply Fubini's theorem.
Then by the formula between the beta function and the gamma function
together with the duplication formula for the gamma function
we reach (\ref{eq:MKBF}).

In \S \ref{sec:BCarg} we will use (\ref{eq:MKBF}) as the following form:
\begin{Lemma}\label{Lem:MKBF2}
Let $\nu\in\C$. Then in $\Rep(s)>(|\Rep(\nu)|-1)/2$
we have
\begin{equation}\label{eq:MKBF2}
\int_0^{\infty}
 2(\pm iu)^{1/2}K_{\nu}(2(\pm iu)^{1/2})u^s\frac{du}{u}\\
=\Gamma\left(s+\frac{1+\nu}{2}\right)
\Gamma\left(s+\frac{1-\nu}{2}\right)e^{\mp\pi is/2},
\end{equation}
where $(\pm iu)^{1/2}=e^{\pm\pi i/4}u^{1/2}$.
\end{Lemma}
\begin{proof}
We replace $s$ with $2s+1$ on (\ref{eq:MKBF}).
Then we change the variable $u$ by $2u^{1/2}$.
Consequently we have
\[
 \int_0^{\infty}2u^{1/2}K_{\nu}(2u^{1/2})u^s\frac{du}{u}
=\Gamma\left(s+\frac{1+\nu}{2}\right)
\Gamma\left(s+\frac{1-\nu}{2}\right)
\]
in $\Rep(s)>(|\Rep(\nu)|-1)/2$.
Considering (\ref{eq:Knuto0}) and (\ref{eq:Knutoinf}) into account,
we apply the Cauchy theorem, so that
\begin{align*}
 \int_0^{\infty}2u^{1/2}K_{\nu}(2u^{1/2})u^s\frac{du}{u}
&=\int_0^{\pm i\infty}2u^{1/2}K_{\nu}(2u^{1/2})u^s\frac{du}{u}\\
&=e^{\pm \pi is/2}
\int_0^{\infty}2(\pm iu)^{1/2}K_{\nu}(2(\pm iu)^{1/2})u^s\frac{du}{u}.
\end{align*}
These complete the proof.
\end{proof}
 \section{The discussion on $g(x)$ and $g_n$ by Bettin--Conrey}\label{sec:BCarg}
In this section we recall the treatment of $g(x)$ and $g_n$,
which are defined by (\ref{eq:funcg}) and (\ref{eq:gTC}) respectively,
by Bettin and Conrey \cite[\S 5]{BeCo1}.
We start with the fifth displayed formula in \cite[p.5722]{BeCo1}.
Namely,
\begin{equation}\label{eq:BCexp}
 \pi g(1+x)=-\log(2\pi(1+x))
 +\gamma-i\int_{\frac{3}{2}-i\infty}^{\frac{3}{2}+i\infty}
\frac{\zeta(s)\zeta(1-s)}{\sin(\pi s)}(1+x)^{s}ds
\end{equation}
holds.
The integral converges absolutely and locally uniformly in $x>-1$.
To check the convergence, we use the following functional equation
(see \cite[Exercise 8 in p.336]{MoVa})
\begin{equation}\label{eq:RZfteq}
 \zeta(1-s)=2(2\pi)^{-s}\Gamma(s)\cos(\pi s/2)\zeta(s).
\end{equation}
For any fixed $A$, $B$ with $0<A<B$ we have
\begin{equation}\label{eq:AppStirling}
 \Gamma(x+iy)\ll_{A,B}e^{-\pi|y|/2}(|y|+1)^{x-\frac{1}{2}}
\end{equation}
uniformly on $\{x+iy:A\leq x\leq B,~y\in\R\}$,
which follows from the Stirling formula
(see \cite[Problem 8 in p.15]{Leb}
or \cite[eq.~(C.19) in p.523]{MoVa}).
Applying (\ref{eq:AppStirling}) and $\cos(\pi s/2)\ll e^{\pi|t|/2}$
to (\ref{eq:RZfteq}),
we find $\zeta(1-s)\ll |t|+1$
for $s=\frac{3}{2}+it$ with $t\in\R$.
We easily see $\zeta(s)\ll 1$ and $1/\sin(\pi s)\ll e^{-\pi|t|}$
for $s=\frac{3}{2}+it$ with $t\in\R$.
Consequently we obtain
\begin{equation}\label{eq:intest32}
 \frac{\zeta(s)\zeta(1-s)}{\sin(\pi s)}\ll e^{-\pi|t|}(|t|+1)
\end{equation}
for $s=\frac{3}{2}+it$ with $t\in\R$,
which imply the absolute and locally uniform convergence of the integral.

Inserting the Taylor expansions for $\log(1+x)$ and $(1+x)^s$ at $x=0$
into (\ref{eq:BCexp}),
we find
\begin{equation}\label{eq:gTE}
\begin{aligned}
 \pi g(1+x)&=-\log(2\pi)+\gamma
 -i\int_{\frac{3}{2}-i\infty}^{\frac{3}{2}+i\infty}
 \frac{\zeta(s)\zeta(1-s)}{\sin(\pi s)}ds\\
 &+\sum_{n=1}^{\infty}
 \left(
   \frac{(-1)^n}{n}
   -i\int_{\frac{3}{2}-i\infty}^{\frac{3}{2}+i\infty}
     \frac{\zeta(s)\zeta(1-s)}{\sin(\pi s)}\binom{s}{n}ds
 \right)
x^n
\end{aligned}
\end{equation}
in $|x|<1$.
Here we interchanged the sum and the integral.
This can be justified as follows.
It suffices to show
\begin{equation}\label{eq:absconvSI}
 \sum_{n=0}^{\infty}
\int_{\frac{3}{2}-i\infty}^{\frac{3}{2}+i\infty}
\left|
\frac{\zeta(s)\zeta(1-s)}{\sin(\pi s)}\binom{s}{n}
\right|
|ds|
|x|^n<\infty
\end{equation}
in $|x|<1$.
Let $n\in\Z_{\geq 0}$
and $s=\frac{3}{2}+it$ with $t\in\R$.
To estimate the binomial coefficients, we take arbitrary
$\varepsilon\in(0,1/100)$. Then Cauchy's integral
formula gives
\[
 \binom{s}{n}
=\frac{1}{2\pi i}\int_{|w|=1-\varepsilon}\frac{(1+w)^{s}}{w^{n+1}}dw,
\]
where $\arg(1+w)\in(-\pi/2,\pi/2)$.
On $|w|=1-\varepsilon$ we notice
$|(1+w)^s|=|1+w|^{3/2}e^{-t\arg(1+w)}\leq 2^{3/2}e^{\pi|t|/2}$.
Thus we have
\[
 \left|\binom{s}{n}\right|
\leq
\frac{1}{2\pi}
\int_{|w|=1-\varepsilon}\frac{2^{3/2}e^{\pi|t|/2}}{(1-\varepsilon)^{n+1}}|dw|
=\frac{2^{3/2}e^{\pi|t|/2}}{(1-\varepsilon)^n}.
\]
Since $\varepsilon\in(0,1/100)$ is arbitrary, we reach
\[
 \left|\binom{s}{n}\right|\leq 2^{3/2}e^{\pi|t|/2}.
\]
Combining this with (\ref{eq:intest32}),
we see
\[
\int_{\frac{3}{2}-i\infty}^{\frac{3}{2}+i\infty}
\left|
\frac{\zeta(s)\zeta(1-s)}{\sin(\pi s)}\binom{s}{n}
\right||ds|
\ll\int_{-\infty}^{\infty}e^{-\pi|t|/2}(|t|+1)dt
\ll 1.
\]
This implies (\ref{eq:absconvSI}) in $|x|<1$,
so that the interchange of the sum and the integral is permitted
in $|x|<1$.

We go back to (\ref{eq:gTE}).
We notice
\[
 (-1)^n\binom{s}{n}=\frac{1}{n!}\frac{\Gamma(n-s)}{\Gamma(-s)}.
\]
Thus, for any $n\in\Z_{\geq 1}$ we have\footnote{%
There is a misprint in \cite[p.5722]{BeCo1}.
That is, on the second displayed formula from the bottom,
$\frac{1}{n}-I_n$ should be $\frac{1}{n}+I_n$.
}
\[
 g_n=\frac{1}{n}+\mathcal{I}_n,
\]
where
\[
 \mathcal{I}_n=-\frac{i}{n!}
 \int_{\frac{3}{2}-i\infty}^{\frac{3}{2}+i\infty}
\frac{\zeta(s)\zeta(1-s)}{\sin(\pi s)}\frac{\Gamma(n-s)}{\Gamma(-s)}ds.
\]
By the reflection formula $1/\sin(\pi s)=\Gamma(s)\Gamma(1-s)/\pi$
and (\ref{eq:RZfteq}) we have
\[
 \mathcal{I}_n=\frac{2}{n!\pi i}
\int_{\frac{3}{2}-i\infty}^{\frac{3}{2}+i\infty}
(2\pi)^{-s}\Gamma(s)^2\Gamma(n-s)\cos\left(\frac{\pi s}{2}\right)
\zeta(s)^2\frac{\Gamma(1-s)}{\Gamma(-s)}ds.
\]
We notice
\[
 \Gamma(s)\frac{\Gamma(1-s)}{\Gamma(-s)}
 =\Gamma(s)\times(-s)=-\Gamma(s+1).
\]
This gives
\[
 \mathcal{I}_n=-\frac{2}{n!\pi i}
\int_{\frac{3}{2}-i\infty}^{\frac{3}{2}+i\infty}
(2\pi)^{-s}\Gamma(s)\Gamma(s+1)\Gamma(n-s)\cos\left(\frac{\pi s}{2}\right)
\zeta(s)^2ds.
\]
We insert $\zeta(s)^2=\sum_{m=1}^{\infty}d(m)m^{-s}$ in $\Rep(s)>1$.
We interchange the sum over $m$ and the integral on $\Rep(s)=3/2$.
This can be justified by
$\Gamma(s)\Gamma(s+1)\Gamma(n-s)\cos(\pi s/2)\ll_n e^{-\pi|t|}(|t|+1)^{n+1}$
on $s=\frac{3}{2}+it$ with $t\in\R$ for fixed $n$,
which follows from (\ref{eq:AppStirling}) and $\cos(\pi s/2)\ll e^{\pi|t|/2}$.
In consequence we obtain
\begin{equation}\label{eq:InSE}
 \mathcal{I}_n=-\frac{4}{n!}\sum_{m=1}^{\infty}d(m)\mathcal{I}_n(m),
\end{equation}
where
\[
 \mathcal{I}_n(x)=\frac{1}{2\pi i}
 \int_{\frac{3}{2}-i\infty}^{\frac{3}{2}+i\infty}
 (2\pi x)^{-s}\Gamma(s)\Gamma(s+1)\Gamma(n-s)\cos\left(\frac{\pi s}{2}\right)
 ds.
\]

Bettin and Conrey connected $\mathcal{I}_n(x)$ 
with the $K$-Bessel function $K_1(z)$ to
investigate its behavior.
The corrected version of their formula is
(see fifth displayed formula in \cite[p.5723]{BeCo1})
\begin{Lemma}\label{Lem:InBessel}
For any $n\in\Z_{\geq 2}$ and $x\in\R_{>0}$ we have
\begin{equation}\label{eq:InBessel}
 \mathcal{I}_n(x)
=2x^{-n}
\Rep\left(
\sqrt{2\pi i}\int_0^{\infty}
e^{-u/x}K_1(2\sqrt{2\pi iu})u^{n-\frac{1}{2}}du
\right).
\end{equation}
\end{Lemma}
\begin{proof}
Keeping $\cos(\pi s/2)=(e^{\pi is/2}+e^{-\pi is/2})/2$ in mind,
we firstly show
\begin{equation}\label{eq:InBessel1}
\begin{aligned}
&\frac{1}{2\pi i}\int_{\frac{3}{2}-i\infty}^{\frac{3}{2}+i\infty}
\Gamma(s)\Gamma(s+1)\Gamma(n-s)e^{\pm\pi is/2}y^{-s}ds\\
&=y^{-n}\int_0^{\infty}2e^{-u/y}(\mp iu)^{1/2}K_1(2(\mp iu)^{1/2})u^n\frac{du}{u}
\end{aligned}
\end{equation}
for any $n\in\Z_{\geq 2}$ and $y\in\R_{>0}$.
We see from Lemma \ref{Lem:MKBF2} with $\nu=1$ that
the integral on the left is
\begin{equation}\label{eq:InBessel2}
 \begin{aligned}
  &\frac{1}{2\pi i}\int_{\frac{3}{2}-i\infty}^{\frac{3}{2}+i\infty}
\Gamma(s)\Gamma(s+1)\Gamma(n-s)e^{\pm\pi is/2}y^{-s}ds\\
&=\frac{1}{2\pi i}
\int_{\frac{3}{2}-i\infty}^{\frac{3}{2}+i\infty}
\left(
\int_0^{\infty}2(\mp iu)^{1/2}K_1(2(\mp iu)^{1/2})u^s\frac{du}{u}
\right)
\Gamma(n-s)y^{-s}ds.
 \end{aligned}
\end{equation}
The asymptotic formulas (\ref{eq:Knuto0}) and (\ref{eq:Knutoinf})
with $\nu=1$ give
\[
 2(\mp iu)^{1/2}K_1(2(\mp iu)^{1/2})u^{s-1}\ll
\begin{cases}
 u^{1/2} & \text{if $0<u\leq 1$,}\\
 u^{3/4}e^{-\sqrt{2u}} & \text{if $u\geq 1$,}
\end{cases}
\]
uniformly on $s=\frac{3}{2}+it$ with $t\in\R$.
From the bound (\ref{eq:AppStirling}) we see
$\Gamma(n-s)y^{-s}\ll_{n,y}e^{-\pi|t|/2}(|t|+1)^{n-2}$.
These imply that the integral on the right-hand side
of (\ref{eq:InBessel2}) is absolutely convergent.
In consequence, by Fubini's theorem (\ref{eq:InBessel2}) equals
\begin{equation}\label{eq:InBessel3}
 =\int_0^{\infty}2(\mp iu)^{1/2}K_1(2(\mp iu)^{1/2})
\left(
\frac{1}{2\pi i}\int_{\frac{3}{2}-i\infty}^{\frac{3}{2}+i\infty}
\Gamma(n-s)\left(\frac{y}{u}\right)^{-s}ds
\right)
\frac{du}{u}.
\end{equation}
We know
\begin{equation}\label{eq:InBessel4}
 \frac{1}{2\pi i}\int_{\frac{3}{2}-i\infty}^{\frac{3}{2}+i\infty}
\Gamma(n-s)v^{-s}ds=v^{-n}e^{-1/v}
\end{equation}
for any $n\in\Z_{\geq 2}$ and $v\in\R_{>0}$.
This can be easily checked by
moving the contour $\Rep(s)=3/2$ to $\Rep(s)=+\infty$.
Inserting (\ref{eq:InBessel4}) with $v=y/u$ into (\ref{eq:InBessel3}),
we obtain (\ref{eq:InBessel1}).

We replace $y=2\pi x$ on (\ref{eq:InBessel1}) and change the variable
$u\mapsto 2\pi u$. By (\ref{eq:RPBessel}) with $\nu=1$ we notice
\begin{align*}
&(2\pi iu)^{1/2}K_1(2(2\pi iu)^{1/2})
+(-2\pi iu)^{1/2}K_1(2(-2\pi iu)^{1/2})\\
&=2\Rep\left((2\pi iu)^{1/2}K_1(2(2\pi iu)^{1/2})\right)
\end{align*}
for any $u\in\R_{>0}$.
Combining this with $e^{\pi is/2}+e^{-\pi is/2}=2\cos(\pi s/2)$,
we obtain the claimed formula (\ref{eq:InBessel}).
\end{proof}
Next we confirm that the sum over $m\geq 2$ on (\ref{eq:InSE})
can be negligible in the asymptotic expansion for $\mathcal{I}_n$.
Strictly speaking, we shall show
\begin{Lemma}\label{Lem:Inest}
 As $n\to\infty$ we have
\[
 \mathcal{I}_n=-\frac{4}{n!}\mathcal{I}_n(1)
+O\left(n^{-3/4}e^{-2\sqrt{8\pi n/5}}\right).
\]
\end{Lemma}
To prove this, we give the following
uniform bounds for $\mathcal{I}_n(x)$.
\begin{Lemma}\label{Lem:InUE}
The following two estimates hold uniformly on
$n\in\Z_{\geq 1000}$ and $x\in\R_{\geq 1}$:
\begin{gather}
\mathcal{I}_n(x)\ll\left(\frac{n^2}{e^2\pi x}\right)^n,\label{eq:InUnest1}\\
\mathcal{I}_n(x)
\ll x^{-3/4}n^{n-\frac{3}{4}}e^{-(1+\frac{1}{1000})n}
+x^{1/4}n^{n-\frac{1}{4}}e^{-n}
e^{-2\sqrt{9\pi xn/10}}\label{eq:InUnest2}.
\end{gather}
\end{Lemma}
\begin{proof}
We begin with the integral on (\ref{eq:InBessel}).
We divide the interval $(0,\infty)$ into $(0,1]$ and $[1,\infty)$.
By (\ref{eq:Knuto0}) with $\nu=1$,
the integral over $(0,1]$ is estimated by
\begin{equation}\label{eq:InUE0}
 \int_0^1 e^{-u/x}K_1(2\sqrt{2\pi i u})u^{n-\frac{1}{2}}du
\ll\int_0^1 u^{n-1}du\ll\frac{1}{n}.
\end{equation}
We apply (\ref{eq:Knutoinf}) to the integral over $[1,\infty)$.
Since $|e^{-2\sqrt{2\pi in}}|=e^{-2\sqrt{\pi n}}$,
we find
\begin{equation}\label{eq:InUE1}
\int_1^{\infty} e^{-u/x}K_1(2\sqrt{2\pi i u})u^{n-\frac{1}{2}}du
\ll \int_1^{\infty}e^{-u/x}e^{-2\sqrt{\pi u}}u^{n-\frac{3}{4}}du.
\end{equation}
Below we estimate the last integral by two ways to prove (\ref{eq:InUnest1})
and (\ref{eq:InUnest2}).

First of all we show (\ref{eq:InUnest1}).
We apply the trivial estimate $e^{-u/x}\leq 1$ for any $u\in\R_{\geq 0}$.
Then we change the variable $u$ by $u^2/(4\pi)$, so that
(\ref{eq:InUE1}) is bounded above by
$\leq 2(4\pi)^{-n-\frac{1}{4}}\Gamma(2n+\frac{1}{2})$.
By the Stirling formula we notice $\Gamma(2n+\frac{1}{2})\ll(2n)^{2n}e^{-2n}$.
Thus we obtain
\begin{equation}\label{eq:InUE1-5}
 \int_1^{\infty}e^{-u/x}K_1(2\sqrt{2\pi iu})u^{n-\frac{1}{2}}du
\ll\left(\frac{n^2}{\pi e^2}\right)^n.
\end{equation}
Applying (\ref{eq:InUE0}) and (\ref{eq:InUE1-5}) to (\ref{eq:InBessel}),
we obtain the first estimate (\ref{eq:InUnest1}).

Next we prove (\ref{eq:InUnest2}).
We set $t=t_n=n-\frac{3}{4}$.
When we regard $e^{-u/x}u^t$ as a function of $u$,
its critical point is $u=xt$.
For this reason we change the variable $u$ by $xtu$, so that
\begin{equation}\label{eq:InUE2}
\begin{aligned}
 \int_1^{\infty}e^{-u/x}e^{-2\sqrt{\pi u}}u^{t} du
&=(xt)^{t+1}\int_{1/(xt)}^{\infty}e^{-tu}e^{-2\sqrt{\pi xt u}}u^tdu\\
&=(xt)^{t+1}e^{-t}
\int_{1/(xt)}^{\infty}e^{-t(u-1-\log u)}e^{-2\sqrt{\pi xtu}}du.
\end{aligned}
\end{equation}
We take arbitrary $\varepsilon\in(0,1/10]$.
We divide the interval of the interval into
$[1/(xt),1-\varepsilon)$, $[1-\varepsilon,1+\varepsilon)$
and $(1+\varepsilon,\infty)$.
We consider the integral over $[1/(xt),1-\varepsilon)$.
Since $u-1-\log u$ is monotonically decreasing in $u\in(0,1)$,
the integral is bounded above by
\begin{align*}
 \int_{1/(xt)}^{1-\varepsilon}e^{-t(u-1-\log u)}e^{-2\sqrt{\pi xtu}}du
&\leq e^{-t(-\varepsilon-\log(1-\varepsilon))}\int_0^{\infty}e^{-2\sqrt{\pi xtu}}du\\
&\ll e^{-\frac{\varepsilon^2 t}{2}}\times\frac{1}{xt}.
\end{align*}
On the last inequality we used $-\varepsilon-\log(1-\varepsilon)=\sum_{k=2}^{\infty}\varepsilon^k/k\geq \varepsilon^2/2$ and changed the variable $u$ by $u/(4\pi xt)$.
Next we deal with the integral over $[1+\varepsilon,\infty)$.
Since $u-1-\log u$ is monotonically increasing in $(1,\infty)$,
we see
\begin{align*}
\int_{1+\varepsilon}^{\infty}e^{-t(u-1-\log u)}e^{-2\sqrt{\pi xtu}}du
&\leq e^{-t(\varepsilon-\log(1+\varepsilon))}
\int_{0}^{\infty}e^{-2\sqrt{\pi xtu}}du\\
&\ll e^{-(\frac{\varepsilon^2}{2}-\frac{\varepsilon^3}{3})t}
\times\frac{1}{xt}.
\end{align*}
On the last inequality we used the Taylor expansion of $\log(1+\varepsilon)$.
Finally we consider the integral over $[1-\varepsilon,1+\varepsilon]$.
Changing the variable $u$ by $u+1$, we see
\begin{equation}\label{eq:InUE3}
 \int_{1-\varepsilon}^{1+\varepsilon}e^{-t(u-1-\log u)}e^{-2\sqrt{\pi xtu}}du
=
\int_{-\varepsilon}^{\varepsilon}
e^{-t(u-\log(1+u))}e^{-2\sqrt{\pi xt(1+u)}}du.
\end{equation}
By the Taylor expansion we notice $u-\log(1+u)\geq u^2/10$
on $|u|\leq 1/10$.
In addition, we estimate $e^{-2\sqrt{\pi xt(1+u)}}$ trivially.
In consequence, (\ref{eq:InUE3}) is bounded above by
\[
 \leq e^{-2\sqrt{\pi xt(1-\varepsilon)}}
\int_{-\varepsilon}^{\varepsilon}e^{-tu^2/10}du
\ll \frac{e^{-2\sqrt{\pi xt(1-\varepsilon)}}}{\sqrt{t}}.
\]
Inserting the above bounds into (\ref{eq:InUE2}), we obtain
\begin{equation}\label{eq:InUE4}
\int_1^{\infty}e^{-u/x}e^{-2\sqrt{\pi u}}u^{t} du
\ll (xt)^{t+1}e^{-t}
\left(\frac{e^{-(\frac{\varepsilon^2}{2}-\frac{\varepsilon^3}{3})t}}{xt}
+\frac{e^{-2\sqrt{\pi xt(1-\varepsilon)}}}{\sqrt{t}}\right).
\end{equation}
We specialize $\varepsilon=1/20$.
Then we have
$\frac{\varepsilon^2}{2}-\frac{\varepsilon^3}{3}\geq \frac{1}{1000}$.
We also notice $t^{t+1}\leq n^{n+\frac{1}{4}}$.
For any $n\in\Z_{\geq 1000}$ we find $t=n-\frac{3}{4}\geq \frac{999}{1000}n$,
so that
$t(1-\varepsilon)\geq \frac{999}{1000}\times\frac{19}{20}n\geq \frac{9}{10}n$.
Applying these to (\ref{eq:InUE4}) and keeping (\ref{eq:InUE1}) into mind,
we conclude
\begin{equation}\label{eq:InUE5}
\begin{aligned}
&\int_1^{\infty}e^{-u/x}K_1(2\sqrt{2\pi iu})u^{n-\frac{1}{2}}du\\
&\ll x^{n+\frac{1}{4}}n^{n+\frac{1}{4}}e^{-n}
\left(\frac{e^{-n/1000}}{xn}
+\frac{e^{-2\sqrt{9\pi xn/10}}}{\sqrt{n}}\right).
\end{aligned}
\end{equation}
Inserting (\ref{eq:InUE0}) and (\ref{eq:InUE5}) into (\ref{eq:InBessel}),
we obtain (\ref{eq:InUnest2}).
\end{proof}
\begin{proof}[Proof of Lemma \ref{Lem:Inest}]
We set $N=n^{10}$ and divide the sum into $m=1$,
$2\leq m\leq N$ and $m>N$.
First of all we deal with the sum over $2\leq m\leq N$.
We insert (\ref{eq:InUnest2}).
Rankin's trick gives
\[
 \sum_{m=2}^N d(m)m^{-3/4}\leq
\sum_{m=2}^N d(m)m^{-3/4}\times \frac{N}{m}
\leq N\zeta\left(\tfrac{7}{4}\right)^2.
\]
Since
\begin{align*}
 e^{-2\sqrt{9\pi mn/10}}
&=e^{-2\sqrt{4\pi mn/5}}\times
e^{-2(\sqrt{\frac{9}{10}}-\sqrt{\frac{4}{5}})\sqrt{\pi mn}}\\
&\leq e^{-2\sqrt{8\pi n/5}}\times
e^{-2(\sqrt{\frac{9}{10}}-\sqrt{\frac{4}{5}})\sqrt{\pi m}}
\end{align*}
holds for any $m\in\Z_{\geq 2}$, we find
\[
 \sum_{m=2}^N d(m)m^{1/4}e^{-2\sqrt{9\pi mn/10}}
\ll e^{-2\sqrt{8\pi n/5}}.
\]
In consequence we obtain
\begin{equation}\label{eq:Inest1}
 \begin{aligned}
\sum_{m=2}^Nd(m)\mathcal{I}_n(m)
&\ll Nn^{n-\frac{3}{4}}e^{-(1+\frac{1}{1000})n}
+n^{n-\frac{1}{4}}e^{-n}e^{-2\sqrt{8\pi n/5}}\\  
&\ll n^{n-\frac{1}{4}}e^{-n}e^{-2\sqrt{8\pi n/5}}.
 \end{aligned}
\end{equation}
Next we consider the sum over $m>N$.
We apply (\ref{eq:InUnest1}).
We notice
\[
 \sum_{m=N+1}^{\infty}d(m)m^{-n}
\leq N^{-(n-2)}\sum_{m=N+1}^{\infty}d(m)m^{-2}
\ll N^{-(n-2)}.
\]
This gives
\begin{equation}\label{eq:Inest2}
 \sum_{m=N+1}^{\infty}d(m)\mathcal{I}_n(m)
\ll\left(\frac{n^2}{e^2\pi}\right)^n N^{-(n-2)}
=n^{20}\left(\frac{1}{e^2\pi n^8}\right)^n.
\end{equation}
By (\ref{eq:Inest1}) and (\ref{eq:Inest2})
together with the Stirling formula
$n!\asymp n^{n+\frac{1}{2}}e^{-n}$ we obtain
\[
 -\frac{4}{n!}\sum_{m=2}^{\infty} d(m)\mathcal{I}_n(m)
\ll n^{-3/4}e^{-2\sqrt{8\pi n/5}}.
\]
Inserting this into (\ref{eq:InSE}), we reach the result.
\end{proof}
\section{Confluent hypergeometric functions $U(\alpha;\beta;z)$\\
and their relations with $\mathcal{I}_n(x)$}\label{sec:CHF}
In this section we firstly recall
the confluent hypergeometric functions
$U(\alpha;\beta;z)$ of the second kind.
After that, we point out that $\mathcal{I}_n(x)$ is expressed in terms of
$U(\alpha;\beta;z)$.

Let $\alpha$, $\beta\in\C$ with $\Rep(\alpha)>0$.
The confluent hypergeometric function $U(\alpha;\beta;z)$ is defined by
\begin{equation}\label{eq:defCHF2nd}
 U(\alpha;\beta;z)=\frac{1}{\Gamma(\alpha)}
\int_0^{\infty}e^{-zt}t^{\alpha-1}(1+t)^{\beta-\alpha-1}dt.
\end{equation}
The integral converges absolutely and uniformly on any compact subsets
of $\Rep(z)>0$, so that $U(\alpha;\beta;z)$ is holomorphic in $\Rep(z)>0$.
The function $U(\alpha;\beta;z)$ is holomorphically continued to
$z\in\C\setminus\R_{\leq 0}$ by moving the contour $(0,\infty)$
to $\{te^{i\theta_0}:t\in(0,\infty)\}$
for suitable $\theta_0\in(-\pi/2,\pi/2)$.\footnote{%
If $\beta\notin\Z$, then $U(\alpha;\beta;z)$ can be expressed
in terms of the confluent hypergeometric functions $\CHF$:
see \cite[\S 9.11]{Leb}. This gives the analytic continuation
of $U(\alpha;\beta;z)$.
Since we mainly use $U(\alpha;0;z)$ in this paper,
we give the analytic continuation by moving the contour.
}

We consider the case $\beta=0$ as well as $\Rep(\alpha)>0$.
Then (\ref{eq:defCHF2nd}) turns to
\begin{equation}\label{eq:CHF2nd-0}
 U(\alpha;0;z)=\frac{1}{\Gamma(\alpha)}\int_0^{\infty}e^{-zt}t^{\alpha-1}(1+t)^{-\alpha-1}dt.
\end{equation}
The integral is uniformly convergent on $\Rep(z)\geq 0$.
This implies that (\ref{eq:CHF2nd-0}) is valid on $\Rep(z)=0$ as well as
in $\Rep(z)>0$.
We connect the integral on (\ref{eq:InBessel}) with $U(\alpha;0;z)$
as follows:
\begin{Proposition}\label{Prop:KBLMT}
 Suppose $\Rep(\alpha)>0$. In $z\in\C\setminus\R_{\leq 0}$
we have
\begin{equation}\label{eq:KBLMT}
 \int_0^{\infty}e^{-u}\sqrt{zu}K_1(2\sqrt{zu})u^{\alpha-1}du
=\frac{1}{2}\Gamma(\alpha)\Gamma(\alpha+1)U(\alpha;0;z),
\end{equation}
where the branch of $\sqrt{zu}$ is determined by $\arg(zu)\in(-\pi,\pi)$.
\end{Proposition}
\begin{proof}
 At first we restrict $z$ to $z\in\R_{>0}$.
By (\ref{eq:KnuIE-2}) we have
\begin{align*}
 K_1(2\sqrt{zu})
&=\frac{\sqrt{zu}}{2}\int_0^{\infty}e^{-v-\frac{zu}{v}}v^{-2}dv
=\frac{1}{2}\left(\frac{u}{z}\right)^{1/2}
\int_0^{\infty}e^{-zv-\frac{u}{v}}v^{-2}dv.
\end{align*}
Here in the last equality we changed the variable $v\mapsto zv$.
Inserting this into the left-hand side of (\ref{eq:KBLMT}), we find
\begin{equation}\label{eq:KBLMT1}
 \int_0^{\infty}e^{-u}\sqrt{zu}K_1(2\sqrt{zu})u^{\alpha-1}du
=\frac{1}{2}\int_0^{\infty}e^{-u}u^{\alpha}
\left(\int_0^{\infty}e^{-zv-\frac{u}{v}}v^{-2}dv\right)
du.
\end{equation}
Thanks to $z>0$ and $\Rep(\alpha)>0$ we can check that the integral
is absolute integrable as follows:
\begin{align*}
\int_0^{\infty}\int_0^{\infty}
\left|e^{-u}u^{\alpha}e^{-zv-\frac{u}{v}}v^{-2}\right|dudv
&=\int_0^{\infty}e^{-zv}v^{-2}\int_0^{\infty}e^{-(1+\frac{1}{v})u}u^{\Rep(\alpha)}dudv\\
&=\Gamma(\Rep(\alpha)+1)\int_0^{\infty}e^{-zv}v^{-2}\left(1+\frac{1}{v}\right)^{-\Rep(\alpha)-1}dv\\
&=\Gamma(\Rep(\alpha)+1)\int_0^{\infty}e^{-zv}v^{\Rep(\alpha)-1}(1+v)^{-\Rep(\alpha)-1}dv<\infty.
\end{align*}
By Fubini's theorem we can interchange the order of the integrals,
so that (\ref{eq:KBLMT1}) equals
\begin{align*}
&=\frac{1}{2}\int_0^{\infty}e^{-zv}v^{-2}
\int_0^{\infty}e^{-(1+\frac{1}{v})u}u^{\alpha}du dv
=\frac{1}{2}\Gamma(\alpha+1)\int_0^{\infty}e^{-zv}v^{-2}\left(1+\frac{1}{v}\right)^{-\alpha-1}dv\\
&=\frac{1}{2}\Gamma(\alpha+1)\int_0^{\infty}e^{-zv}v^{\alpha-1}(1+v)^{-\alpha-1}dv.
\end{align*}
By (\ref{eq:KBLMT}) we reach the claimed formula under $z\in\R_{>0}$.

We relax the condition $z\in\R_{>0}$ by the identity theorem.
We fix any $\delta\in(0,\pi/10]$.
We see from (\ref{eq:Knuto0}) and (\ref{eq:Knutoinf}) that $wK_1(w)$
is bounded in $\{w\in\C\setminus\{0\}:|\arg w|\leq\frac{\pi}{2}-\delta\}$.
This implies
$
e^{-u}\sqrt{zu}K_1(2\sqrt{zu})u^{\alpha-1}\ll_{\delta}e^{-u}u^{\Rep(\alpha)-1}
$
uniformly on $\mathcal{S}_{\delta}=\{z\in\C\setminus\{0\}:|\arg z|\leq\pi-\delta\}$ and $u\in\R_{>0}$.
Since $\int_0^{\infty}e^{-u}u^{\Rep(\alpha)-1}du<\infty$,
the integral on (\ref{eq:KBLMT})
is uniformly convergent on $z\in\mathcal{S}_{\delta}$.
This implies the left-hand side of (\ref{eq:KBLMT}) is holomorphic in
$z\in\C\setminus\R_{\leq 0}$.
As was explained in the beginning of this section, $U(\alpha;0;z)$
is holomorphic in $z\in\C\setminus\R_{\leq 0}$.
Thus by the identity theorem we replace the condition $z\in\R_{>0}$
by $z\in\C\setminus\R_{\leq 0}$.
This completes the proof.
\end{proof}
Applying Proposition \ref{Prop:KBLMT} to (\ref{eq:InBessel}), we see
\begin{Corollary}\label{Cor:InCHF}
 For any $n\in\Z_{\geq 2}$ and $x\in\R_{>0}$ we have
\begin{equation}\label{eq:InxCHF}
 \mathcal{I}_n(x)
=(n-1)!n!\Rep(U(n;0;2\pi ix)).
\end{equation}
\end{Corollary}
\begin{proof}
 We change the variable $u\mapsto ux$ on (\ref{eq:InBessel}).
Applying (\ref{eq:KBLMT}) with $\alpha=n$ and $z=2\pi i$,
we obtain (\ref{eq:InxCHF}).
\end{proof}
\section{Proof of Theorem \ref{Thm:main}}\label{sec:PMT}
In this section we prove Theorem \ref{Thm:main}.
By Lemma \ref{Lem:Inest} and Corollary \ref{Cor:InCHF}
it suffices to give an asymptotic expansion for $(n-1)!U(n;0;2\pi i)$
as $n\to\infty$.
This is one of large parameter problems, whose answer has been
known for more general settings.
See \cite[\S 6.13.2]{ErdHTF1} for historical accounts.
In order to compute the asymptotic coefficients explicitly,
we reconstruct it along with \cite[\S10.3.2]{Tem}, restricting
the discussion to our case.

To state the asymptotic formula, we fix the notation.
We set
\[
 \lambda(u)=\frac{1}{e^u-1}-\frac{1}{u}+\frac{1}{2}.
\]
The function $\lambda(u)$ is holomorphic in
$\C\setminus\{2\pi im:m\in\Z\setminus\{0\}\}$.
We easily notice
\begin{equation}\label{eq:lambdaTE}
 \lambda(u)=\sum_{k=1}^{\infty}\frac{B_{k+1}}{(k+1)!}u^k
\end{equation}
in $|u|<2\pi$.
We define $P_k(z)$ by the generating function as
\[
 e^{-z\lambda(u)}=\sum_{k=0}^{\infty}P_k(z)u^k.
\]
We set $P_k=P_k(2\pi i)$.
We put
\begin{equation}\label{eq:deftC}
 \wC_l
=\sum_{\begin{subarray}{c}
        j,k\geq 0\\
	j+k=l
       \end{subarray}}
(k+1,j)(2\pi i)^kP_k 2^{-2j}.
\end{equation}
As will be seen in Lemma \ref{Lem:Pkexp} and Corollary \ref{Cor:Pkexp2},
$P_0(z)=1$ and $z^k P_k(z)\in z^2\Q[z^2]$
hold for any $k\in\Z_{\geq 1}$.
This implies $(2\pi i)^k P_k\in\Q[\pi^2]$ for any $k\in\Z_{\geq 0}$,
so that $\wC_l\in\Q[\pi^2]$ holds for any $l\in\Z_{\geq 0}$.

With this notation we have
\begin{Proposition}\label{Prop:In1asym}
 We fix any $L\in\Z_{\geq 1}$.
Then as $n\to\infty$ we have
\begin{align*}
-\frac{4}{n!}\mathcal{I}_n(1)
&=2^{9/4}\pi^{3/4}e^{-2\sqrt{\pi n}}
\sum_{l=0}^L(2\pi)^{-l/2}\wC_ln^{-\frac{l}{2}-\frac{3}{4}}
\sin\left(2\sqrt{\pi n}+\frac{\pi l}{4}+\frac{3\pi}{8}\right)
+O_L\left(n^{-\frac{L}{2}-\frac{5}{4}}e^{-2\sqrt{\pi n}}\right).
\end{align*}
\end{Proposition}
Below we prove Proposition \ref{Prop:In1asym}.
By Corollary \ref{Cor:InCHF} and (\ref{eq:CHF2nd-0}) we find
\begin{equation}\label{eq:InCHF2nd-1}
 \begin{aligned}
-\frac{4}{n!}\mathcal{I}_n(1)
&=-4(n-1)!\Rep(U(n;0;2\pi i))
=-4\Rep\left(
\int_0^{\infty}e^{-2\pi it}t^{n-1}(1+t)^{-n-1}dt
\right).
 \end{aligned}
\end{equation}
Changing the variable $t=1/(e^u-1)$, we notice
\begin{equation}\label{eq:InCHF2nd-2}
-4\int_0^{\infty}e^{-2\pi it}t^{n-1}(1+t)^{-n-1}dt
=4\int_0^{\infty}e^{-2\pi i\lambda(u)}e^{-nu-\frac{2\pi i}{u}}du.
\end{equation}

To investigate asymptotic properties of the integral on (\ref{eq:InCHF2nd-2}),
we define $r_L(u)$ by
\begin{equation}\label{eq:defRF}
 e^{-2\pi i\lambda(u)}=\sum_{k=0}^L P_k u^k+r_L(u).
\end{equation}
Then $r_L(u)$ is holomorphic in $\C\setminus\{2\pi im:m\in\Z\setminus\{0\}\}$.
By definition we have the bound $r_L(u)\ll_L |u|^{L+1}$ on $|u|\leq 1$.
Since $\lambda(u)$ is bounded on $\mathcal{T}=\{u\in\C:|u|\geq 1,~|\arg u|\leq\pi/4\}$,
so is $e^{-2\pi i\lambda(u)}$.
This implies $r_L(u)\ll_L|u|^{L}\ll|u|^{L+1}$
on $u\in\mathcal{T}$.
Consequently we obtain
\begin{equation}\label{eq:bdRF}
 r_L(u)\ll|u|^{L+1}
\end{equation}
on the sector $\{u\in\C:|\arg u|\leq\pi/4\}$.

Inserting (\ref{eq:defRF}) into (\ref{eq:InCHF2nd-2}), we have
\begin{equation}\label{eq:InCHF2nd-3}
 \begin{aligned}
&-4\int_0^{\infty}e^{-2\pi it}t^{n-1}(1+t)^{-n-1}dt
=4\sum_{k=0}^L
P_k\int_0^{\infty}e^{-nu-\frac{2\pi i}{u}}u^kdu
+4\int_0^{\infty}r_L(u)e^{-nu-\frac{2\pi i}{u}}du.
 \end{aligned}
\end{equation}
We deal with the last integral.
We see from the Cauchy theorem together with $|e^{-\frac{2\pi i}{u}}|\leq 1$
on $\{u\in\C:0\leq\arg u\leq\pi/4\}$ that
\begin{equation}\label{eq:bdintRF1}
\begin{aligned}
&\int_0^{\infty}r_L(u)e^{-nu-\frac{2\pi i}{u}}du
=\int_0^{\infty e^{\pi i/4}}r_L(u)e^{-nu-\frac{2\pi i}{u}}du
=e^{\pi i/4}\int_0^{\infty}r_L(ue^{\pi i/4})
\exp\left(-nue^{\pi i/4}-\frac{2\pi}{u}e^{\pi i/4}\right)du.
\end{aligned}
\end{equation}
By (\ref{eq:bdRF}) the absolute value of (\ref{eq:bdintRF1})
is bounded above by
\begin{equation}\label{eq:bdintRF2}
 \ll_L \int_0^{\infty}u^{L+1}e^{-\frac{nu}{\sqrt{2}}-\frac{\sqrt{2}\pi}{u}}du.
\end{equation}
Changing the variable $u\mapsto\sqrt{2}u/n$ and applying (\ref{eq:KnuIE-2}),
we see that (\ref{eq:bdintRF2}) is
\[
 \ll_L\frac{1}{n^{L+2}}\int_0^{\infty}e^{-u-\frac{\pi n}{u}}u^{L+1}du
\ll n^{-\frac{L}{2}-1}K_{-(L+2)}(2\sqrt{\pi n}).
\]
Applying (\ref{eq:Knutoinf}), we obtain
\begin{equation}\label{eq:bdintRF3}
 \left|
\int_0^{\infty}r_L(u)e^{-nu-\frac{2\pi i}{u}}du
\right|
\ll_L n^{-\frac{L}{2}-\frac{5}{4}}e^{-2\sqrt{\pi n}}.
\end{equation}
We treat the first term on the right-hand side of (\ref{eq:InCHF2nd-3}).
By (\ref{eq:KnuIE-2}) and (\ref{eq:KBesselef}) we find
\begin{equation}\label{eq:mainInCHF1}
 \int_0^{\infty}e^{-nu-\frac{2\pi i}{u}}u^kdu
=2(2\pi i)^{\frac{k+1}{2}}n^{-\frac{k+1}{2}}K_{k+1}(2\sqrt{2\pi in}).
\end{equation}
We see from (\ref{eq:Knutoinf}) with $\nu=k+1$ and $J=L$
that (\ref{eq:mainInCHF1}) is
\[
 =\pi^{1/2}e^{-2\sqrt{2\pi in}}
\left(\sum_{j=0}^L(k+1,j)2^{-2j}(2\pi i)^{\frac{k-j}{2}+\frac{1}{4}}
n^{-\frac{j+k}{2}-\frac{3}{4}}
+O_{k,L}
\left(n^{-\frac{L+k}{2}-\frac{5}{4}}\right)\right).
\]
Inserting this into the first term on the right-hand side of
(\ref{eq:InCHF2nd-3}), we find
\begin{equation}\label{eq:mainInCHF2}
 \begin{aligned}
&4\sum_{k=0}^L P_k\int_0^{\infty}e^{-nu-\frac{2\pi i}{u}}u^kdu\\
= \,
&4\pi^{1/2}e^{-2\sqrt{2\pi in}}
\sum_{k=0}^L\sum_{j=0}^L
(k+1,j)2^{-2j}P_k(2\pi i)^{\frac{k-j}{2}+\frac{1}{4}}n^{-\frac{j+k}{2}-\frac{3}{4}}
+O_L\left(n^{-\frac{L}{2}-\frac{5}{4}}e^{-2\sqrt{\pi n}}\right).  
 \end{aligned}
\end{equation}
For each $l\in\{0,1,\ldots,L\}$ we collect the terms labelled by
$(j,k)$ satisfying $j+k=l$.
Consequently (\ref{eq:mainInCHF2}) equals
\begin{equation}\label{eq:mainInCHF3}
 =4\pi^{1/2}e^{-2\sqrt{2\pi in}}
\sum_{l=0}^L(2\pi i)^{-\frac{l}{2}+\frac{1}{4}}\wC_l
n^{-\frac{l}{2}-\frac{3}{4}}
+O_L\left(n^{-\frac{L}{2}-\frac{5}{4}}e^{-2\sqrt{\pi n}}\right).
\end{equation}

Applying (\ref{eq:bdintRF3}) and (\ref{eq:mainInCHF3})
to (\ref{eq:InCHF2nd-3}), we reach
\begin{Lemma}\label{Lem:Inasym}
 For each $L\in\Z_{\geq 1}$ we have
\begin{align*}
&-4\int_0^{\infty}e^{-2\pi it}t^{n-1}(1+t)^{-n-1}dt
=4\pi^{1/2}e^{-2\sqrt{2\pi in}}
\sum_{l=0}^L(2\pi i)^{-\frac{l}{2}+\frac{1}{4}}\wC_l
n^{-\frac{l}{2}-\frac{3}{4}}
+O_L\left(n^{-\frac{L}{2}-\frac{5}{4}}e^{-2\sqrt{\pi n}}\right).
\end{align*}
\end{Lemma}

Next we investigate properties of $\wC_l$.
For this purpose we firstly show the following expression
for $P_k(z)$.
\begin{Lemma}\label{Lem:Pkexp}
 We have $P_0(z)=1$ and
\begin{equation}\label{eq:Pkexp}
 P_k(z)=\sum_{l=1}^k(-1)^l\frac{z^l}{l!}
\sum_{\begin{subarray}{c}
       k_1,\ldots,k_l\geq 1\\
       k_1+\cdots+k_l=k
      \end{subarray}}
\frac{B_{k_1+1}\cdots B_{k_l+1}}{(k_1+1)!\cdots(k_l+1)!}
\end{equation}
for any $k\in\Z_{\geq 1}$.
\end{Lemma}
\begin{proof}
The Taylor expansion for the exponential function gives
\[
 e^{-z\lambda(u)}=1+\sum_{l=1}^{\infty}(-1)^l\frac{z^l}{l!}\lambda(u)^l
\]
in $|u|<2\pi$ and $z\in\C$.
We see from (\ref{eq:lambdaTE}) that this equals
\begin{align*}
 &=1+\sum_{l=1}^{\infty}(-1)^l
\frac{z^l}{l!}
\sum_{k_1=1}^{\infty}\cdots\sum_{k_l=1}^{\infty}
\frac{B_{k_1+1}\cdots B_{k_l+1}}{(k_1+1)!\cdots(k_l+1)!}\\
&=1+\sum_{l=1}^{\infty}(-1)^l
\frac{z^l}{l!}
\sum_{k=l}^{\infty}
\sum_{\begin{subarray}{c}
       k_1,\ldots,k_l\geq 1\\
       k_1+\cdots+k_l=k
      \end{subarray}}
\frac{B_{k_1+1}\cdots B_{k_l+1}}{(k_1+1)!\cdots(k_l+1)!}.
\end{align*}
Changing the order of the sums over $k$ and $l$,
we obtain the claimed result.
\end{proof}
\begin{Corollary}\label{Cor:Pkexp1}
 Let $k\in\Z_{\geq 1}$. Then $P_k(z)\in\Q[z]$ with degree $k$.
Its leading term is
$(-1)^k/(k!12^k)$.
\end{Corollary}
\begin{proof}
 This immediately follows from Lemma \ref{Lem:Pkexp} and $B_2=1/6$.
\end{proof}
\begin{Corollary}\label{Cor:Pkexp2}
 Let $k\in\Z_{\geq 1}$. Then we have
\[
 P_k(z)\in
\begin{cases}
\langle z,z^3,\ldots,z^k\rangle_{\Q}   & \text{if $k$ is odd,}\\
\langle z^2,z^4,\ldots,z^k\rangle_{\Q} & \text{if $k$ is even.}
\end{cases}
\]
\end{Corollary}
\begin{proof}
We note that $B_m$ vanishes if $m\geq 3$ is an odd integer.
Thus the sum over $(k_1,\ldots,k_l)$ on (\ref{eq:Pkexp}) vanishes
if $k\equiv l+1\pmod{2}$.
This completes the proof. 
\end{proof}
\begin{Corollary}\label{Cor:Pkexp3}
 We have $\wC_0=1$ and
\begin{equation}\label{eq:Pkexp3}
 \wC_l\in\langle\pi^{2m}:m=0,1,\ldots,l\rangle_{\Q}(\subset \R)
\end{equation}
for each $l\in\Z_{\geq 1}$.
Its coefficients of $\pi^0$ and $\pi^{2l}$ are $(1,l)\times 2^l$
and $1/(l!3^l)$, respectively.
In particular, $\wC_l$ does not vanish for each $l\in\Z_{\geq 0}$.
\end{Corollary}
\begin{proof}
 We can easily check $\wC_0=1$.
By Corollary \ref{Cor:Pkexp2} we notice
\[
 (2\pi i)^k P_k\in
\begin{cases}
 \langle\pi^{k+1},\pi^{k+3},\ldots,\pi^{2k}\rangle_{\Q}
& \text{if $k$ is odd,}\\
 \langle\pi^{k+2},\pi^{k+4},\ldots,\pi^{2k}\rangle_{\Q}
& \text{if $k$ is even}
\end{cases}
\]
for any $k\in\Z_{\geq 1}$.
This implies (\ref{eq:Pkexp3}).
The $\pi^0$-term appears only at $(j,k)=(l,0)$
on (\ref{eq:deftC}).
This gives that the coefficient of $\pi^0$ is $(1,l)2^{-2l}$.
The $\pi^{2l}$-term appears only at $(j,k)=(0,l)$.
This and Corollary \ref{Cor:Pkexp1} yield that the coefficient
of $\pi^{2l}$ equals
\[
 (l+1,0)\times(2i)^l\times\frac{(-1)^l}{l!}\frac{1}{12^l}\times(2i)^l
=\frac{1}{l!3^l}.
\]
Since $\pi$ is transcendental, the nonvanishing of the $\pi^0$-coefficient
implies $\wC_l\neq 0$.
This completes the proof.
\end{proof}
\begin{proof}[Proof of Proposition \ref{Prop:In1asym}]
 We apply Lemma \ref{Lem:Inasym} to (\ref{eq:InCHF2nd-1}).
Considering Corollary \ref{Cor:Pkexp3} and $\sqrt{2\pi in}=\sqrt{\pi n}(1+i)$
into account, we find
\begin{align*}
-\frac{4}{n!}\mathcal{I}_n(1)
&=4\pi^{1/2}e^{-2\sqrt{\pi n}}
\sum_{l=0}^L(2\pi)^{-\frac{l}{2}+\frac{1}{4}}\wC_l
\Rep\left(e^{-i(2\sqrt{\pi n}+\frac{\pi l}{4}-\frac{\pi}{8})}\right)
n^{-\frac{l}{2}-\frac{3}{4}}
+O\left(n^{-\frac{L}{2}-\frac{5}{4}}e^{-2\sqrt{\pi n}}\right).
\end{align*}
Since
\[
 \Rep\left(e^{-i(2\sqrt{\pi n}+\frac{\pi l}{4}-\frac{\pi}{8})}\right)
=\cos\left(2\sqrt{\pi n}+\frac{\pi l}{4}-\frac{\pi}{8}\right)
=\sin\left(2\sqrt{\pi n}+\frac{\pi l}{4}+\frac{3\pi}{8}\right),
\]
we obtain the result.
\end{proof}
Summarizing the discussion, we conclude
\begin{Theorem}
 The asymptotic formula (\ref{eq:main}) holds for each $L\in\Z_{\geq 1}$,
where $\wC_l$ are given by (\ref{eq:deftC}).
In particular, Theorem \ref{Thm:main} is true.
\end{Theorem}
\begin{proof}
 Inserting Proposition \ref{Prop:In1asym} into Lemma \ref{Lem:Inest},
we obtain (\ref{eq:main}).
We have already seen $\wC_l\in\langle\pi^{2m}:m=0,1,\ldots,l\rangle_{\Q}$
and $\wC_l\neq 0$ in Corollary \ref{Cor:Pkexp3}.
\end{proof}
\section{A remark on numerical computation of $g_n$}\label{sec:RNC}
We can compute $g_n$ numerically via (\ref{eq:gnexpr}) in principle.
On this computation we need to pay attention to a problem
on loss of significance as described below.

First of all we observe $b_k$ for $k\in\Z_{\geq 2}$.
We easily notice $b_k=0$ for any odd $k\in\Z_{\geq 3}$.
On the other hand, $\zeta(2l)=(-1)^{l-1}2^{2l-1}\pi^{2l}B_{2l}/(2l)!$
for $l\in\Z_{\geq 1}$ (see \cite[Corollary B.3 in p.500]{MoVa}) gives
\[
 b_{2l}=(-1)^{l-1}\frac{2^{2l-1}\pi^{2l}B_{2l}^2}{2l\cdot(2l)!}
=(-1)^{l-1}\frac{(2l)!}{l\cdot(2\pi)^{2l}}\zeta(2l)^2.
\]
By the latter expression and the Stirling formula we find
\[
 b_{2l}\sim(-1)^{l-1}2\sqrt{\frac{\pi}{l}}\left(\frac{l}{\pi e}\right)^{2l}
\]
as $l\to\infty$.
We see from this and Theorem \ref{Thm:main} that
$g_n$ is much smaller than the absolute value of a single term
labelled by even $j$ on (\ref{eq:gnexpr}).
In other words, large cancellation occurs in (\ref{eq:gnexpr}).

To compute numerical values of $g_n$, we observe
\begin{gather*}
 G_n^{(1)}=\frac{1}{n(n+1)}+2|b_n|+2\sum_{j=0}^{n-2}\binom{n-1}{j}|b_{j+2}|,\\
G_n^{(\infty)}=2\max_{0\leq j\leq n-2}\binom{n-1}{j}|b_{j+2}|
\end{gather*}
for $n\in\Z_{\geq 2}$.
We can easily see that $\{G_n^{(1)}\}$ and $\{G_n^{(\infty)}\}$
are monotonically increasing.
Numerical values of $G_n^{(1)}$ and $G_n^{(\infty)}$
are given in Table \ref{Tab:NumGn}.
When we compute $g_n$ via (\ref{eq:gnexpr}),
significant digits of precision have to be much larger than
the decimal digits of $G_n^{(1)}$.
\begin{table}
\caption{Numerical values of $G_n^{(1)}$ and $G_n^{(\infty)}$}
\label{Tab:NumGn}
\begin{center}
 \begin{tabular}{c||c|c}
 $n$     & $G_n^{(1)}$                   & $G_n^{(\infty)}$ \\ \hline
$51$     & $4.53\cdots\times 10^{27}$    & $1.46\cdots\times 10^{27}$    \\
$101$    & $3.63\cdots\times 10^{81}$    & $1.16\cdots\times 10^{81}$    \\
$501$    & $1.32\cdots\times 10^{739}$   & $4.25\cdots\times 10^{738}$   \\
$1001$   & $7.14\cdots\times 10^{1773}$  & $2.29\cdots\times 10^{1773}$  \\
$5001$   & $7.19\cdots\times 10^{12339}$ & $2.30\cdots\times 10^{12339}$ \\
$10001$  & $1.22\cdots\times 10^{27683}$ & $3.91\cdots\times 10^{27682}$ \\
$50001$  & $4.59\cdots\times 10^{173333}$& $1.47\cdots\times 10^{173333}$\\
$100001$ & $7.86\cdots\times 10^{376761}$& $2.52\cdots\times 10^{376761}$
 \end{tabular}
\end{center}
\end{table}

When we create Figure \ref{Fig:main2}, we compute numerical values
of $g_n$
for $2\leq n\leq 10001$ with $30000$ decimal digits.
The computation was done by Pari/GP and Figure \ref{Fig:main2}
was drawn by matplotlib library in Python 3.


\begin{thebibliography}{99}
 \bibitem{BeCo1}
S.~Bettin and J.~B.~Conrey,
A reciprocity formula for a cotangent sum, 
Int.~Math.~Res.~Not.~IMRN {\bf 2013} no.~24 (2013) 5709–5726.
\bibitem{BeCo2}
S.~Bettin and B.~Conrey,
Period functions and cotangent sums,
Algebra Number Theory {\bf 7} (2013) 215--242.
\bibitem{BFOR}
K.~Bringmann, A.~Folsom, K.~Ono and L.~Rolen,
Harmonic Maass forms and mock modular forms:
theory and applications,
American Mathematical Society Colloquium Publications {\bf 64},
American Mathematical Society,
Providence, RI, 2017, xv+391 pages.
\bibitem{ErdHTF1}
A.~Erd\'{e}lyi,
Higher Transcendental Functions, Vol.~I,
McGraw-Hill Book Company, New York, 1953, xxvi+302 pages.
\bibitem{ErdHTF2}
A.~Erd\'{e}lyi,
Higher Transcendental Functions, Vol.~II,
McGraw-Hill Book Company, New York, 1953, xvi+396 pages.
 \bibitem{ErdTIT1}
A.~Erd\'{e}lyi, Tables of integral transforms, Vol.~I,
Bateman Manuscript Project, McGraw-Hill Book Company,
New York, 1954, xx+391 pages.
\bibitem{Fo}
A.~Folsom,
Twisted Eisenstein series, cotangent-zeta sums,
and quantum modular forms,
Trans.~London Math.~Soc. {\bf 7} (2020) 33--48.
\bibitem{Leb}
N.~N.~Lebedev, Special functions and their applications,
Dover Publications, New York, 1972, xii+308 pages.
\bibitem{LeZa}
J.~Lewis and D.~Zagier,
Period functions for Maass wave forms. I,
Ann.~of Math.~(2) {\bf 153} (2001) 191--258.
\bibitem{MoVa}
H.~L.~Montgomery and R.~C.~Vaughan,
Multiplicative number theory.~I.~Classical theory,
Cambridge Studies in Advanced Mathematics {\bf 97},
Cambridge University Press, Cambridge, 2007, xvii+552 pages.
\bibitem{Tem}
N.~M.~Temme,
Asymptotic Methods for Integrals,
Series in Analysis {\bf 6},
World Scientific, Singapore, 2015, xxii+605 pages.
\bibitem{Zag1}
D.~Zagier,
Quantum modular forms,
Quanta of Maths, Clay Math.~Proc.~{\bf 11},
Amer.~Math.~Soc., Providence, RI, 2010, 659--675.
\end{thebibliography}
\end{document}